\documentclass[12pt]{amsart}

\usepackage{amsmath}
\usepackage{amsfonts}
\usepackage{amssymb}
\usepackage{amsthm}
\usepackage{hyperref}
\usepackage{enumerate}
\usepackage{cleveref}
\usepackage{mathrsfs}
\usepackage[top=0.9in, bottom=1.15in, left=1.1in, right=1.1in]{geometry}
 \usepackage{hyperref}
\hypersetup{colorlinks=true,linkcolor=blue,citecolor=magenta}
\newtheorem{theorem}{Theorem}

\newtheorem{corollary}[theorem]{Corollary}
\newtheorem{definition}[theorem]{Definition}
\newtheorem*{remark}{Remark}

\Crefname{conjecture}{Conjecture}{Conjectures}

\theoremstyle{plain}

\theoremstyle{plain}

\author{Robert Schneider} 
\address{Department of Mathematics and Computer Science\newline
Emory University\newline
400 Dowman Dr., W401\newline
Atlanta, Georgia 30322}
\email{robert.schneider@emory.edu}


\title{Alternating ``strange'' functions}
 
\begin{document}

\begin{abstract}
In this note we consider infinite series similar to the ``strange'' function $F(q)$ of Kontsevich studied by Zagier, Bryson-Ono-Pitman-Rhoades, Bringmann-Folsom-Rhoades, Rolen-Schneider, and others in connection to quantum modular forms. We show that a class of ``strange'' alternating series that are well-defined almost nowhere in the complex plane can be added (using a modified definition of limits) to familiar infinite products to produce convergent $q$-hypergeometric series, of a shape that specializes to Ramanujan's mock theta function $f(q)$, Zagier's quantum modular form $\sigma(q)$, and other interesting number-theoretic objects. We also discuss Ces\`{a}ro sums for these alternating series, and continued fractions that are similarly ``strange''.
\end{abstract}

\maketitle

\section{Introduction and statement of results}\label{Sect1}
In a 1997 lecture at the Max Planck Institute for Mathematics, 
Fields medalist Maxim Kontsevich discussed an almost nonsensical $q$-hypergeometric series \cite{Zagier_Vassiliev}
\begin{equation}\label{F(q)def}
F(q):=\sum_{n=0}^{\infty}(q;q)_n,
\end{equation}  
where the {\it $q$-Pochhammer symbol} is defined by $(a;q)_0:=1$, $(a;q)_n:=\prod_{j=0}^{n-1}(1-aq^j)$, and $(a;q)_{\infty}:=\lim_{n\to \infty}(a;q)_n$ for $a,q\in \mathbb C, |q|<1$. This series $F(q)$ is often referred to in the literature as {\it Kontsevich's ``strange'' function}, and has since been studied deeply by Zagier \cite{Zagier_Vassiliev} --- it was one of his prototypes for quantum modular forms, which enjoy beautiful transformations similar to classical modular forms, and also resemble objects in quantum theory \cite{Zagier_quantum} --- as well as by 
other authors \cite{BFR,BOPR,RolenSchneider} in connection to quantum modularity, unimodal sequences, and other topics. 

There are many reasons to say 
the series (\ref{F(q)def}) is ``strange'' (see \cite{Zagier_Vassiliev}). For brevity, 
let us merely note that as $n\to \infty$, then $(q;q)_{n}$ converges on the unit disk, is essentially singular on the unit circle (except at roots of unity, where it vanishes),  
and diverges when $|q|>1$. Thus $\sum_{n\geq 0} (q;q)_n$ converges {almost nowhere} 
in the complex plane. 
However, at $q=\zeta_m$ an $m$th order root of unity, $F$ is suddenly very well-behaved: because $(\zeta_m;\zeta_m)_{n}=0$ for $n\geq m$, then as $q\to \zeta_m$ radially,  $F(\zeta_m):=\lim_{q\to\zeta_m} F(q)$ is just a polynomial in $\mathbb Z[\zeta_m]$. 

Now let us turn our attention to the alternating case of this series, viz. 
\begin{equation}
\widetilde{F}(q):=\sum_{n=0}^{\infty}(-1)^n(q;q)_{n},
\end{equation}
a summation that has been studied by Cohen \cite{BOPR}, which is similarly ``strange'': it doesn't converge anywhere in $\mathbb C$ except at roots of unity, where it is a polynomial. 
In fact, 
computational examples suggest the odd and even partial sums of $\widetilde{F}(q)$ oscillate asymptotically between two convergent $q$-series. 

To capture this oscillatory behavior, let us adopt a 
notation we will use throughout. If $S$ is an infinite series, we will write $S_+$ to denote the limit of the sequence of odd partial sums, and $S_-$ for the limit of the even partial sums, if these limits exist (clearly if $S$ converges, then $S_+=S_-=S$). 

Interestingly, like $F(q)$, the ``strange'' series $\widetilde{F}(q)$ is closely connected to a sum Zagier provided as another prototype for quantum modularity (when multiplied by $q^{1/24}$) \cite{Zagier_quantum}, the function 
\begin{equation}\label{sigma}
\sigma(q):=\sum_{n=0}^{\infty}\frac{q^{n(n+1)/2}}{(-q;q)_n}=1+\sum_{n=0}^{\infty}(-1)^n q^{n+1}(q;q)_n
\end{equation}
from Ramanujan's ``lost'' notebook, with the right-hand equality due to Andrews \cite{AJO}. If we use the convention introduced above and write $\widetilde{F}_{+}(q)$ (resp. $\widetilde{F}_{-}(q)$) to denote the limit of the odd (resp. even) partial sums of $\widetilde{F}$, we can state this connection explicitly, depending on the choice of ``$+$'' or ``$-$''.

\begin{theorem}\label{thm1}
For $0<|q| < 1$ we have 
\begin{equation*}
\sigma(q)=2\widetilde{F}_{\pm}(q)\pm (q;q)_{\infty}.
\end{equation*}
\end{theorem}
%
%


We can make further sense of alternating ``strange'' series such as this using Ces\`{a}ro summation, a well-known alternative definition of the limits of infinite series 
(see \cite{Hardy_divergent}). 

\begin{definition} The Ces\`{a}ro sum of an infinite series is the limit of the arithmetic mean of successive partial sums, if the limit exists.
\end{definition}

In particular, it follows immediately that the Ces\`{a}ro sum of the series $S$ is the average $\frac{1}{2}(S_+ + S_-)$ if the limits $S_+,S_-$ exist. Then Theorem \ref{thm1} leads to the following fact.

\begin{corollary}
We have that $\frac{1}{2}\sigma(q)$ is the Ces\`{a}ro sum of the ``strange'' function $\widetilde{F}(q)$.
\end{corollary}

A similar relation to Theorem \ref{thm1} involves Ramanujan's prototype $f(q)$ for a mock theta function 
\begin{equation}\label{f(q)def}
f(q):=\sum_{n=0}^{\infty}\frac{q^{n^2}}{(-q;q)_n^{2}}=1-\sum_{n=1}^{\infty}\frac{(-1)^{n}q^n }{(-q;q)_n},
\end{equation} 
the right-hand side of which is due to Fine (see (26.22) in \cite{Fine}, Ch. 3). 
Now, if we define 
\begin{equation}
\widetilde{\phi}(q):=\sum_{n=0}^{\infty}\frac{(-1)^n}{(-q;q)_n},
\end{equation}
which is easily seen to be ``strange'' like the previous cases, and write  $\widetilde{\phi}_{\pm}$ for limits of the odd/even partial sums as above, we can write $f(q)$ in terms of the ``strange'' series and an infinite product.

\begin{theorem}\label{thm2}
For $0<|q|<1$ 
we have 
\begin{equation*}
f(q)=2\widetilde{\phi}_{\pm}(q)\pm \frac{1}{(-q;q)_{\infty}}.
\end{equation*}
\end{theorem}

Again, the Ces\`{a}ro sum results easily from this theorem.

\begin{corollary}
We have that $\frac{1}{2}f(q)$ is the Ces\`{a}ro sum of the ``strange'' function $\widetilde{\phi}(q)$.
\end{corollary}


%



Theorems \ref{thm1} and \ref{thm2} typify a general phenomenon: 
the combination of an alternating Kontsevich-style ``strange'' function 
with a related infinite product is a convergent $q$-series when we fix the $\pm$ sign in this modified definition of limits. 
Let us fix a few more notations in order to discuss this succinctly. As usual, we write 
\begin{equation*}
(a_1,a_2,...,a_r;q)_n:=(a_1;q)_n (a_2;q)_n \cdots (a_r;q)_n,
\end{equation*} 
along with the limiting case $(a_1,a_2,...,a_r;q)_{\infty}$ as $n\to\infty$. 
Associated to the sequence $a_1,a_2,...,a_r$ of complex coefficients, we will define a polynomial $\alpha_r(X)$ 
by the relation
\begin{equation}
(1-a_1 X)(1-a_2 X)\cdots (1-a_r X)=: 1-\alpha_r(X) X,
\end{equation}
thus
\begin{equation}
(a_1q,a_2q,...,a_rq;q)_n=\prod_{j=1}^{n}(1-\alpha_r(q^j)q^j), 
\end{equation}
and we follow this convention in also writing $(1-b_1 X)(1-b_2 X)\cdots (1-b_s X)=: 1-\beta_s(X) X$ for complex coefficients $b_1,b_2,...,b_s$. Moreover, 
we define a generalized alternating ``strange'' series:
\begin{equation}
\widetilde{\Phi}(a_1,a_2,...,a_r;b_1,b_2,...,b_s; q):=\sum_{n=0}^{\infty} (-1)^n \frac{(a_1q,a_2q,...,a_rq;q)_{n}}{(b_1q,b_2q,...,b_sq;q)_{n}}
\end{equation}
Thus $\widetilde{F}(q)$ is the case $\widetilde{\Phi}(1;0; q)$, and $\widetilde{\phi}(q)$ is the case $\widetilde{\Phi}(0;-1; q)$. 
We note that if $q$ is a $k$th root of $1/a_i$ for some $i$, then $\widetilde{\Phi}$ truncates after $k$ terms like $F$ and $\widetilde{F}$. As above, let $\widetilde{\Phi}_{\pm}$ denote the limit of the odd/even partial sums; then we can encapsulate the preceding theorems in the following statement.

\begin{theorem}\label{thm3}
For $0<|q|<1$ 
we have
\begin{flalign*}
2\widetilde{\Phi}_{\pm}&(a_1,a_2,...,a_r;b_1,b_2,...,b_s; q) \pm\frac{(a_1q,a_2q,...,a_rq;q)_{\infty}}{(b_1q,b_2q,...,b_sq;q)_{\infty}}\\
&=1-\sum_{n=1}^{\infty} \frac{(-1)^n q^{n}\left(\alpha_r(q^n)-\beta_s(q^n)\right) (a_1q,a_2q,...,a_rq;q)_{n-1}}{(b_1q,b_2q,...,b_s q;q)_{n}}.
\end{flalign*}
\end{theorem}

From this identity we can fully generalize the previous corollaries.

\begin{corollary}
We have that $1/2$ times the right-hand side of Theorem \ref{thm3} is the Ces\`{a}ro sum of the ``strange'' function $\widetilde{\Phi}(a_1,...,a_r;b_1,...,b_s; q)$.
\end{corollary}


The takeaway is that the $N$th partial sum of an alternating ``strange'' series oscillates asymptotically as $N\to \infty$ between $\frac{1}{2}(S(q)+(-1)^N P(q))$, where $S$ is an Eulerian infinite series and $P$ is an infinite product as given in Theorem \ref{thm3}. 
We recover Theorem \ref{thm1} from Theorem \ref{thm3} as the case $a_1=1,\  a_i=b_j=0$ for all $i>1,j\geq 1$. Theorem \ref{thm2} is the case $b_1=-1,\  a_i=b_j=0$ for all $i\geq1, j>1$. 

Considering these connections together with diverse connections made by Kontsevich's $F(q)$ to important objects of study \cite{BFR, BOPR,Zagier_Vassiliev}, it seems the ephemeral 
``strange'' functions almost ``enter into mathematics as beautifully''\footnote{To redirect Ramanujan's words} as their 
convergent 
(but still eccentric) 
relatives, mock theta functions. 

%

\begin{remark}
It follows from Euler's continued fraction formula \cite{Euler} that 
alternating ``strange'' functions have 
representations such as 
\begin{equation*}
\widetilde{F}(q)=\frac{{1}}{1+\frac{1-q}{q+\frac{1-q^2}{q^2+\frac{1-q^3}{q^3+...}}}},\  \  \  \  \widetilde{\phi}(q) =  \frac{{1}}{1+\frac{1}{q+\frac{1+q}{q^2+\frac{1+q^2}{q^3+...}}}}.
\end{equation*}
These ``strange'' continued fractions diverge on $0<|q|<1$ with successive convergents equal to the corresponding partial sums of the series representation. We can substitute continued fractions for the Kontsevich-style summations in the theorems if we give a similarly modified definition of convergence; for example, we can write $$f(q)=\frac{{2}}{1+\frac{1}{q+\frac{1+q}{q^2+\frac{1+q^2}{q^3+...}}}}\pm \frac{1}{(-q;q)_{\infty}}$$ where we take the $\pm$ sign to be positive if we define the limit of the continued fraction to be the limit of the even convergents, and negative if instead we use odd convergents.
\end{remark}

\section{Proofs of results}

In this section we quickly prove the preceding theorems, 
and justify the corollaries.

\begin{proof}[Proof of Theorem \ref{thm1}]
Using telescoping series to find that
\begin{flalign*}
(q;q)_{\infty}
= 1-\sum_{n=0}^{\infty}(q;q)_n \left( 1 - (1-q^{n+1})\right)
= 1-\sum_{n=0}^{\infty}q^{n+1}(q;q)_n,
\end{flalign*}
and combining this functional equation with the right side of (\ref{sigma}) above, easily gives
\begin{equation*}
\sigma(q)-(q;q)_{\infty}=2\sum_{n=0}^{\infty}q^{2n+1}(q;q)_{2n}.
\end{equation*}
On the other hand, manipulating symbols heuristically (for we are working with a divergent series $\widetilde{F}$) 
suggests we can rewrite 
\begin{flalign*}
\widetilde{F}(q)=\sum_{n=0}^{\infty}\left( (q;q)_{2n}-(q;q)_{2n+1}\right)&=\sum_{n=0}^{\infty}(q;q)_{2n}\left(1- (1-q^{2n+1})\right)=\sum_{n=0}^{\infty}q^{2n+1}(q;q)_{2n},
\end{flalign*}
which is a rigorous statement 
if by convergence on the left we mean the limit as $N\to\infty$ of partial sums $\sum_{n=0}^{2N-1}(-1)^n (q;q)_n$. 
We can also choose the alternate coupling of summands to similar effect, e.g. considering here the partial sums $1+\sum_{n=1}^{N-1}[(q;q)_{2n}-$ $(q;q)_{2n-1}]-(q;q)_{2N-1}$ as $N\to \infty$. 
Combining the above considerations 
proves the 
theorem for $|q|<1$, which one finds to agree with computational examples. 
\end{proof}

\begin{proof}[Proof of Theorem \ref{thm2}]
Following the 
formal steps that prove Theorem \ref{thm1} 
above, we can use 
\begin{equation*}
\frac{1}{(-q;q)_{\infty}}=1-\sum_{n=0}^{\infty}\frac{1}{(-q;q)_n}\left( 1 - \frac{1}{1+q^{n+1}}\right)=1-\sum_{n=1}^{\infty}\frac{q^n}{(-q;q)_n}
\end{equation*}
and rewrite the related ``strange'' series 
\begin{equation*}
\widetilde{\phi}(q)=\sum_{n=0}^{\infty} \frac{1}{(-q;q)_{2n}}\left( 1-\frac{1}{1+q^{2n+1}}\right)=\sum_{n=0}^{\infty} \frac{q^{2n+1}}{(-q;q)_{2n+1}},
\end{equation*}
which of course fails to converge for $0<|q|<1$ on the left-hand side 
but makes sense if we use the modified definition of convergence 
in Section \ref{Sect1}, to yield the identity in the theorem 
(which is, again, borne out by computational examples).
\end{proof}

\begin{proof}[Proof of Theorem \ref{thm3}]
Using the definitions of the polynomials $\alpha_r(X),\beta_s(X)$, then following the exact steps that yield Theorems \ref{thm1} and \ref{thm2}, i.e., manipulating and comparing telescoping-type series with the same modified definition of convergence, 
gives the theorem. 
\end{proof}

\begin{proof}[Proof of corollaries]
Clearly, for an alternating ``strange'' series in which the odd and even partial sums each approach a different limit, the average of these two limits will equal the Ces\`{a}ro sum of the series.
\end{proof}

\end{document}